\documentclass[12pt]{amsart}
\numberwithin{equation}{section}

\usepackage{amssymb}
\def\ca{{\mathcal A}}
\def\cb{{\mathcal B}}

\def\ce{{\mathcal E}}
\def\cf{{\mathcal F}}

\def\ch{{\mathcal H}}

\def\ck{{\mathcal K}}

\def\cs{{\mathcal S}}

\def\ga{{\mathfrak A}}

\def\gam{{\mathfrak M}}
\def\gn{{\mathfrak N}}

\def\gx{{\mathfrak X}}
\def\gz{{\mathfrak Z}}

\def\bc{{\mathbb C}}

\def\bg{{\mathbb G}}

\def\bm{{\mathbb M}}
\def\bn{{\mathbb N}}

\def\bt{{\mathbb T}}

\def\bz{{\mathbb Z}}

\def\a{\alpha}
\def\b{\beta}
\def\g{\gamma} \def\G{\Gamma}
\def\d{\delta}  \def\D{\Delta}

\def\eps{\varepsilon}

\def\m{\mu}

\def\n{\nu}
\def\r{\rho}
\def\s{\sigma} \def\S{\Sigma}

\def\f{\varphi} \def\F{\Phi}
  
\def\om{\omega} \def\Om{\Omega}

\def\id{\hbox{id}}

\newtheorem{thm}{Theorem}[section]
\newtheorem{lem}[thm]{Lemma}

\newtheorem{cor}[thm]{Corollary}
\newtheorem{prop}[thm]{Proposition}
\newtheorem{defin}[thm]{Definition}

\newtheorem{con}[thm]{Conjecture}

\def\ad{\mathop{\rm Ad}}
\def\aut{\mathop{\rm Aut}}

\newcommand{\ty}[1]{\mathop{\rm {#1}}}
\def\di{\mathop{\rm d}\!}

\newcommand{\nn}{\nonumber}

\newcommand{\ots}[1]{\otimes_{\mathop{\rm #1}}}

\def\idd{{1}\!\!{\rm I}}
\def\spec{\mathop{\rm spec}}

\begin{document}

\title[diagonal measures]
{a nonconventional ergodic theorem for quantum ``diagonal measures''}
\author{Francesco Fidaleo}
\address{Francesco Fidaleo,
Dipartimento di Matematica,
Universit\`{a} di Roma Tor Vergata, 
Via della Ricerca Scientifica 1, Roma 00133, Italy} \email{{\tt
fidaleo@mat.uniroma2.it}}


\begin{abstract}
We extend the Nonconventional Ergodic Theorem for generic measures by Furstenberg, to several situations of interest arising from quantum dynamical systems. We deal with the diagonal state canonically associated to the product state (i.e. quantum "diagonal measures"), or to convex combinations of diagonal measures for non ergodic cases. For the sake of completeness, we treat also the Nonconventional Ergodic Theorem for compact dynamical systems, that is when the unitary generating the dynamics in the GNS representation is almost
periodic. The Nonconventional Ergodic Theorem allows in a natural way to determine the limit of the three--point correlations, naturally relevant for the knowledge of the ergodic properties of a dynamical system.
\end{abstract}

\maketitle

\section{introduction}

The scope of the present paper is the generalisation to several situations arising from the  
non commutative setting, of the Nonconventional Ergodic Theorem of H. Furstenberg relative 
to diagonal measures (cf. \cite{Fu, Fu1}) for classical ergodic dynamical systems. Namely, we prove an 
ergodic theorem relative to possibly non invariant and non normal states, 
which is the generalisation of Theorem 3.1 of \cite{Fu1} and Theorem 4.2 of \cite{F3}. We deal with ergodic cases, and non ergodic ones as well.

Let $(\ga,\a,\om)$ be a $C^*$--dynamical system based on the  unital $C^*$--algebra 
$\ga$, the automorphism $\a$, and the invariant state $\om$. Suppose that the support $s(\om)$ of the state $\om$ in the bidual $\ga^{**}$, is in the centre $Z(\ga^{**})$. Notice that this condition is trivially satisfied in the classical case. Let
$\big(\ch_\om,\pi_{\om},U_\om,\Om\big)$ be the associated Gelfand--Naimark--Segal (GNS for short) covariant representation. The $C^*$--algebra 
$\gam:=\pi_{\om}(\ga)''\otimes _{\rm{max}}\pi_{\om}(\ga)'$ acts in a natural way on 
$\ch_\om\otimes\ch_\om$, and on $\ch_\om$. Then the vector state 
$$
A\otimes B\in\pi_{\om}(\ga)''\otimes _{\rm{max}}\pi_{\om}(\ga)'\mapsto\langle AB\Om,\Om\rangle
$$ 
is the quantum "diagonal measure" corresponding to the product state 
$$A\otimes B\in\pi_{\om}(\ga)''\otimes _{\rm{max}}\pi_{\om}(\ga)'\mapsto\langle A\Om,\Om\rangle
\langle B\Om,\Om\rangle\,.
$$
In addition, on $\gam$ is naturally acting the automorphism
$$
\g:=\ad{}\!_{U_\om^{k_1}}\otimes\ad{}\!_{U_\om^{k_2}}\,,
$$
where $k_1,k_2\in\bz$ are fixed integers.
The product state is automatically invariant under the action of $\b$. Conversely, the diagonal measure is in general, neither normal with respect to the product measure, nor invariant with respect to the action of $\b$. The nonconventional Ergodic Theorem proved in the present paper concerns the study of the convergence in the strong operator topology of the sequence of the Cesaro means
\begin{equation}
\label{eetuu}
\frac{1}{N}\sum_{n=0}^{N-1}
U_\om^{nk_1}XU_\om^{n(k_2-k_1)}
\end{equation} 
when $X$ belongs to the norm closed linear subspace generated of
$\pi_{\om}(\ga)''$ and $\pi_{\om}(\ga)'$ in $\cb(\ch_\om)$. We treat the ergodic case, that is when 
$U_\om^{n(k_2-k_1)}$ is ergodic, and the non ergodic case, with the additional condition 
\begin{equation}
\label{eetuu1u}
\pi_{\om}(\ga)'\cap\{U_\om^{k_2-k_1}\}'\subset Z(\pi_\om(\ga)'')\,,
\end{equation} 
and some natural restrictions on the integers $k_1,k_2\in\bz$. Notice that also \eqref{eetuu1u} is automatically satisfied in the classical situation. The non ergodic situation relies on reduction theory concerning the decomposition of a state with central support in the bidual into states which have still central support almost everywhere. Such a crucial result is available only in the separable situation. Thus, we restrict the analysis for the non ergodic situation, to separable 
$C^*$--dynamical systems.

The investigation of the limit of the Cesaro means in \eqref{eetuu} provides a nontrivial case for which the Entangled Ergodic Theorem (see \cite{AHO}, see also \cite{F1, F2} and the references cited therein) holds true. All such results concerning the behaviour of the Cesaro Means as before, allow us to investigate in a natural way, the limit of the three--point correlations for quantum dynamical systems. The reader is referred to \cite{AET, F3, F4, F2, NSZ} for the systematic treatment of the topic, and for a wide class of interesting situations. We list below the situations relative to the three--point correlations 
\begin{equation}
\label{tpc}
\left\{\frac1N\sum_{n=0}^{N-1}\om\left(A_0\a^{nk_1}(A_1)\a^{nk_2}(A_2)\right)\right\}
\end{equation}
under consideration in the present paper:
\begin{itemize}
\item[(i)] the support $s(\om)$ in the bidual $\ga^{**}$ is central, and $\om$ is ergodic for 
$\a^{k_2-k_1}$ (i.e. when $\Om_\om$ is the unique invariant vector, up to multiplication for a constant, for $U^{k_2-k_1}_\om$);
\item[(ii)]  for $k_1=kl$, $k_2=(k+1)l$, with central support $s(\om)$ and 
$\pi_\om(\ga)'\cap\{U_\om^l\}'\subset\gz_\om$, when $\ga$ is separable.
\end{itemize}
For the sake of completeness, we also consider the case when
\begin{itemize}
\item[(iii)] $(\ga,\a,\om)$ is compact (i.e. when $\ch_\om$ is generated by the eigenvectors of 
$U_\om$) without any further restriction.
\end{itemize}

\section{preliminaries}

\subsection{Notations}
Let $X$, $Y$ be linear spaces. The algebraic tensor product is 
denoted by $X\odot Y$. If $\ch$ and $\ck$ are Hilbert spaces, the 
Hilbertian tensor product, that is the completion of $\ch\odot \ck$
under the norm induced by the inner product  
$$
\langle x\otimes\xi,y\otimes\eta\rangle:=
\langle x,y\rangle\langle\xi,\eta\rangle\,,
$$
is denoted by $\ch\otimes \ck$.

Let $\{A_{\a}\}_{\a\in J}\subset\cb(\ch)$ be a net consisting of 
bounded operators acting on the Hilbert space $\ch$. If it converges 
to $A\in\cb(\ch)$ in the weak or strong operator topology, we write respectively
$$
\mathop{\rm w\!-\!lim}_{\a}A_{\a}=A\,,\quad
\mathop{\rm s\!-\!lim}_{\a}A_{\a}=A\,.
$$
Let $U$ be a unitary operator acting on $\ch$. Consider the resolution of the identity  
$\{E_U(\D)\,:\,\D\,\,\text{Borel subset of}\,\,\bt\}$ of $U$. Denote  
$E^U_{z}:=E_U(\{z\})$. Namely, $E^U_{z}$ is nothing but the selfadjoint 
projection on the eigenspace corresponding to the eigenvalue 
$z$ of $U$ in the unit circle $\bt$. 

The unitary $U$ is said to be {\it ergodic} if $E^U_{1}\ch$ is one 
dimensional. By von Neumann Mean Ergodic Theorem, being $U$ ergodic is equivalent to the 
existence of a unit vector $\xi_{0}\in\ch$ such that 
$$
\lim_{N\to+\infty}\frac{1}{N}\sum_{n=0}^{N-1}U^{n}\xi=
\langle\xi,\xi_{0}\rangle\xi_{0}\,,\quad \xi\in\ch\,,
$$
or equivalently,
$$
\lim_{N\to+\infty}\frac{1}{N}\sum_{n=0}^{N-1}\langle U^{n}\xi,\eta\rangle=
\langle\xi,\xi_{0}\rangle\langle\xi_{0},\eta\rangle\,,\quad \xi,\eta\in\ch\,.
$$
The unitary $U$ is said to be {\it almost periodic} if
$\ch=\ch_{\mathop{\rm ap}}^{U}$, $\ch_{\mathop{\rm ap}}^{U}$ being 
the closed subspace 
consisting of the vectors having relatively norm--compact orbit 
under $U$. It is seen in \cite{NSZ} that $U$ is almost periodic if and 
only if $\ch$ is generated by the eigenvectors of $U$. The set $\s_{\mathop{\rm pp}}(U)$ of all the eigenvalues of $U$ is referred as the {\it pure point spectrum}. 

In the present paper we consider only unital $C^*$--algebras $\ga$ with the unity $\idd\in\ga$. 
For an algebra $\ca$, $Z(\ca)$ denotes its {\it centre}.
Let $\f\in\ga^*$ be a positive functional. Consider the GNS representation 
$\big(\ch_\f,\pi_{\f},\F\big)$ (cf. \cite{T}, Section I.9)  
canonically associated to $\f$. The centre of the representation is defined as
$$
\gz_\f:=\pi_\f(\ga)''\cap\pi_\f(\ga)'\,.
$$
Denote $s(\f)\in\ga^{**}$ the support of the state $\f$ in the bidual.
It is well known that $s(\f)$ is central, that is $s(\f)\in Z(\ga^{**})$ if and only if 
$\F$ is 
separating for $\pi_\f(\ga)''$, see e.g. \cite{SZ}, Section 10.17. Denote $M:=\pi_\f(\ga)''$. The commutant von Neumann 
algebra is $M'\equiv \pi_\f(\ga)'$. When $s(\f)$ is central, it is possible to introduce the so called Tomita antilinear invoution densely defined as
$$
S_\F:A\F\in M\F\mapsto A^*\F\in\ch_\f\,.
$$
It can be shown that $S_\F$ is closable, with closure denoted again by $S_\F$ with an abuse of notation. We have for the polar decomposition, $S_\F=J_\F\D_\F^{1/2}$, where $\D_\F$ is the Tomita {\it modular operator} and $J_\F$ the Tomita {\it modular conjugation}. We recall the main properties to the modular conjugation used in the sequel: $J_\F\F=\F$, and 
\begin{equation}
\label{jei}
J_\F \pi_\f(\ga)''J_\F=\pi_\f(\ga)'\,, \quad J_\F CJ_\F=C^*\,,\,\,\, C\in\gz_\f\,.
\end{equation}
The reader is referred to \cite{SZ} and the literature cited therein, for the Tomita Modular Theory.

\subsection{Dynamical Systems} 
\label{ss22}
For a $C^*$--dynamical system we mean a triplet 
$\big(\ga,\a,\om\big)$ 
consisting of a $C^*$--algebra $\ga$ which we always suppose to have an identity $\idd$, an action $\a:g\in G\mapsto \a_g\in\aut(\ga)$ of the group $G$ by $*$--automorphisms (denoted simply as {\it automorphisms}) $\a$ of $\ga$, 
and a state $\om\in\cs(\ga)$ invariant under the action of $\a$. Mainly, the actions considered in the sequel are those of $\bz$, that is those determined by a single automorphism. We refer to such dynamical systems as {\it discrete}, or simply $C^*$--dynamical systems when there is no matter of confusion.

A $C^*$--dynamical system $\big(\ga,\a,\om\big)$
is said to be {\it ergodic} if $\om\in\partial\ce_G(\ga)$, that is it is extremal in the convex, compact (in the $*$--weak topology) set $\ce_G(\ga)$ of the invariant states. Denote 
$\big(\ch_\om,\pi_{\om},U_\om,\Om\big)$ the GNS covariant representation 
(cf. \cite{T}, Section I.9)  
canonically associated to the dynamical system under 
consideration. Suppose for simplicity that $G=\bz$. It can be shown that if $s(\om)\in Z(\ga^{**})$, then the state $\om$ is ergodic if and only if
$$
\lim_{N\to+\infty}\frac{1}{N}\sum_{n=0}^{N-1}\om(A\a^{n}(B))=
\om(A)\om(B)\,,\quad A,B\in\ga\,,
$$
or equivalently if and only if $U_\om$ is ergodic. The reader is referred to 
\cite{BR} (cf. Theorem \ref{br1} below) and the reference cited therein for further details. In the general situation, denote $E_\om$ the orthogonal projection onto the subspace of the invariant vectors under the action of $G$:
$$
E_\om\ch_\om=\{\xi\in\ch_\om\mid U_\om(g)\xi=\xi\,,g\in G\}\,.
$$
Even if it is not directly used, we report the following result which has an interest in itself.
\begin{prop}
Let $\om\in\ce_{G}(\ga)$ such that $s(\om)\in Z(\ga^{**})$. If $\pi_\om(\ga)'\cap\{U_\om(G)\}'\subset\gz_\om$ then
$$
\pi_\om(\ga)'\cap\{U_\om(G)\}'=\gz_\om\cap\{U_\om(G)\}'\,.
$$
\end{prop}
\begin{proof}
We start by noticing that $J_\Om$ commutes with $U_\om(g)$, $g\in G$, see \cite{NSZ}, Proposition 3.3. By \eqref{jei}, we obtain
\begin{align*}
&\pi_\om(\ga)''\cap\{U_\om(G)\}'=J_\Om\pi_\om(\ga)'J_\Om\cap\{U_\om(G)\}'\\
=&J_\Om\big(\pi_\om(\ga)'\cap\{U_\om(G)\}'\big)J_\Om
=\pi_\om(\ga)'\cap\{U_\om(G)\}'\,,
\end{align*}
which leads to the assertion.
\end{proof}
The following results (cf. Theorem 4.3.20 of \cite{BR}) are useful in the sequel. 
\begin{thm}
\label{br1}
Let $\om\in\ce_{G}(\ga)$ such that $s(\om)\in Z(\ga^{**})$. Then the following assertions are equivalent.
\begin{itemize}
\item[(i)] $\pi_{\om}(\ga)''\cap\{U_\om(G)\}'=\bc I$,
\item[(ii)] $E_{\om}$ has rank one,
\item[(iii)] $\om\in\partial\ce_{G}(\ga)$,
\item[(iv)] $\pi_{\om}(\ga)'\cap \{U_\om(G)\}'=\bc I$.
\end{itemize}
\end{thm}
A $C^*$--dynamical system $(\ga,\a,\om)$ is said to be $G$--{\it Abelian} if 
$$
E_{\om}\pi_{\om}(\ga)E_{\om}\subset\cb(\ch_\om)
$$ 
is a family of mutually commuting operators.
We end the present section by reporting Proposition 4.3.7 of \cite{BR}.
\begin{thm}
\label{br2}
Let $\om\in\ce_{G}(\ga)$, such that $s(\om)\in Z(\ga^{**})$. Then the following assertions are equivalent.
\begin{itemize}
\item[(i)] $(\ga,\a,\om)$ is $G$--Abelian,
\item[(ii)] $\pi_{\om}(\ga)'\cap \{U_\om(G)\}'$ is an Abelian von Neumann algebra,
\item[(iii)] there exists a unique maximal measure $\n\in M_\om(\ce_{G}(\ga))$.
\end{itemize}
\end{thm}
Notice that $\n$ is precisely the {\it orthogonal measure} uniquely associated to the Abelian von Neumann algebra $\pi_{\om}(\ga)'\cap \{U_\om(G)\}'$ (cf. Theorem 4.1.25). Such a measure is pseudo--supported on the ergodic states 
$\partial\ce_{G}(\ga)$ (cf. \cite{BR}, Proposition 4.3.2). In addition, if $\ga$ is separable, each Borel set is a Baire set, then such a measure $\n$ is supported on the ergodic states 
$\partial\ce_{G}(\ga)$. The reader is referred to Section 4.3 of \cite{BR} for details and proofs.

We now report some crucial results contained in \cite{F3, NSZ} concerning discrete 
$C^*$--dynamical systems. For $v\in\bt$ denote 
$$
M_v=\{A\in M: \a(A)=vA\}\,,\quad M'_v=\{B\in M': \a(B)=vB\}\,.
$$
\begin{prop}
\label{gnsc1}
Let $\big(\ga,\a,\om\big)$ be the $C^*$--dynamical system such that 
$s(\om)\in Z(\ga^{**})$. Then with the previous notations, 
$$
\overline{M_v\Om}=\overline{M'_v\Om}=E^{U_\om}_{v}\ch\,.
$$
In addition, $\s_{\mathop{\rm pp}}(U_\om)=\s_{\mathop{\rm pp}}(U_\om)^{-1}$ and
$J_\Om E^{U_\om}_v=E^{U_\om}_{\bar v}J_\Om$. Suppose further that $\om$ is ergodic, then 
$$
E^{U_\om}_{v}\ch=\bc V_v\Om=\bc W_v\Om\,,\quad v\in\s_{\mathop{\rm pp}}(U_\om)\,,
$$
for some a unique
unitary $V_{v}\in M$ ($W_{v}\in M'$) up to a phase. In addition, $\s_{\mathop{\rm pp}}(U_\om)$ is 
a subgroup of $\bt$.
\end{prop}

\subsection{Nonconventional Ergodic Theorem} 

Let $\big(\gam,\b,\f\big)$ be a $C^*$--dynamical system, together with another state
$\r\in\cs(\ga)$. Notice that $\r$ is supposed in general to be neither invariant 
under the action of $\b$, nor normal w.r.t. $\f$. Let $\gn\subset\gam$ be a 
$*$--subalgebra. Denote $\gn^{\b}:=\gn\cap\gam^{\b}$ the set of the elements of $\gn$ invariant under $\b$. 
\begin{defin}
\label{fu}
The state $\r$ is said to be generic (cf. \cite{F3, Fu})
for $\big(\gam,\b,\f\big)$ 
if there exists a $*$--subalgebra $\gn\subset\gam$ such that
\begin{itemize}
\item[(i)] $\overline{\pi_{\f}(\gn^{\b})\F}=E^{U_\f}_{1}\ch_{\f}$,
\item[(ii)] for each $B\in\gn$,
$$
\lim_{N\to+\infty}\frac{1}{N}\sum_{n=0}^{N-1}\r(\b^{n}(B))=\f(B)\,.
$$
\end{itemize}
\end{defin}
Let the state $\r$ be generic w.r.t. a subalgebra $\gn$. It was shown in \cite{F3} (see \cite{Fu} for the classical case) that 
\begin{equation}
\label{fu4}
\pi_{\f}(B)\F\in\pi_{\f}(\gn^{\b})\F\subset\ch_\f\mapsto \pi_{\r}(B)\rm{P}\in\ch_{\r}
\end{equation}
uniquely defines a partial isometry $V:\ch_{\f}\mapsto\ch_{\r}$ such 
that $V^{*}V=E^{U_\f}_{1}\ch_{\f}$. 
We report Theorem 3.5 of \cite{F3}, which is the extension to the non commutative case of Furstenberg Nonconventional Ergodic Theorem (cf. \cite{Fu, Fu1}).
\begin{thm}
\label{fu1}
Let $\r$ be generic for $\big(\gam,\b,\f\big)$ 
w.r.t. a $*$--subalgebra $\gn$.
Then for each $B\in\gn$,
$$
\lim_{N\to+\infty}\frac{1}{N}\sum_{n=0}^{N-1}\pi_{\r}(\a^{n}(B))\rm{P}
=V\left(\pi_{\f}(B)\F\right)\,,
$$
$V:\ch_{\f}\mapsto\ch_{\r}$ being the partial isometry uniquely defined by 
\eqref{fu4}.
\end{thm}

\section{on the reduction theory}
\label{sec:red}

We recall some general facts concerning the direct integral decomposition of representations and states. We will mainly concern with the separable case if it is not otherwise specified. 

Let $\ga$ be a unital not necessarily separable $C^*$--algebra which we always suppose to have an unity $\idd$. Consider a state 
$\f\in\cs(\a)$. It is then possible to decompose the state
\begin{equation}
\label{dcpz}
\f=\int_{\cs(\ga)}\psi\di\n(\psi)\,,
\end{equation}
where $\n$ is a Radon probability measure on the convex $*$--weakly compact set which is precisely the 
{\it orthogonal measure} (cf. \cite{BR}, Section 4.1.3) associated to $\gz$. In addition,
$\gz\sim L^\infty(\cs(\ga),\m)$. Suppose now that $\f$ is invariant w.r.t. the action $\a$ of a group $G$. In addition, 
if $\gz=\pi_\om(\ga)'\cap\{U_\om(G)\}'$, then $\n$ is supported on the invariant states 
$\ce_G(\ga)$, and pseudo--supported on the erogic states 
$\partial\ce_G(\ga)$. It is also the unique maximal measure in $M_\om(\ce_{G}(\ga))$. We refer the reader to Section 4 of \cite{BR}, and Section 3 of \cite{S}.

We now pass to the case of interest for us, that is when $\ga$ is separable and $G=\bz$, acting on $\ga$ by powers of a single automorphism $\a$. Let $(\ch,\pi)$ be a representation of $\ga$ on a separable Hilbert space, together with an Abelian von Neumann sub algebra $\gz\subset\pi(\ga)'$. Consider the decomposition $\pi=\int^\oplus_Z\pi_z\di\m(z)$ w.r.t. $\gz$ (i.e. such that $\gz$ is the algebra of diagonalisable operators: $\gz\sim L^\infty(Z,\m)$), which always exists according Theorem 8.3.2 of \cite{DX0}. 
\begin{lem}
\label{dixx}
With the above notations, we have
$\pi(\ga)''=\int^{\oplus}_{Z}\pi_{z}(\ga)''\di\m(z)$ if and only if $\gz\subset\pi(\ga)''$.
\end{lem}
\begin{proof}
We can assume that $\ch=\int^{\oplus}_{Z}\ch_z\di\m(z)$. Then
$$
\gz=\left\{\int^{\oplus}_{Z}f(z)I(z)\di\m(z)\mid f\in L^{\infty}(Z,\m)\right\}\,,
$$ 
where $I(z)$ is the identity of $\cb(\ch_z)$. One implication is trivial. and the reverse implication is proved in Lemma
8.4.1 of \cite{DX0}.
\end{proof}
We collect the main results of interest for us in the following
\begin{thm}
\label{minsepr}
Let $(\ga,\a,\om)$ be a $C^*$--dynamical system as above. Suppose that $\ga$ is separable,  
$s(\om)\in\ga^{**}$, and 
$$
\gz=\pi_\om(\ga)'\cap\{U_\om\}'\subset\gz_\om\,.
$$
Consider the orthogonal measure $\n$ corresponding to $\gz$ given in \eqref{dcpz}. The following assertions hold true.
\begin{itemize}
\item[(i)] 
\begin{align*}
\big(\ch_\om,\pi_{\om},U_\om,\Om\big)\cong&
\bigg(\int^{\oplus}_{\partial\ce_\bz(\ga)}\ch_{\psi}\di\n(\psi),
\int^{\oplus}_{\partial\ce_\bz(\ga)}\pi_{\psi}\di\n(\psi),\\
&\int^{\oplus}_{\partial\ce_\bz(\ga)}U_{\psi}\di\n(\psi),
\int^{\oplus}_{\partial\ce_\bz(\ga)}\Psi\di\n(\psi)\bigg)\,.
\end{align*}
\end{itemize}
In addition, there exist a measurable set $F\subset\partial\ce_\bz(\ga)$ of full measure such that for $\psi\in F$,
\begin{itemize}
\item[(ii)] $s(\psi)\in Z(\ga^{**})$, and 
$J_\Om=\int^{\oplus}_{\partial\ce_\bz(\ga)}J_{\Psi}\di\n(\psi)$;
\item[(iii)] $J_{\Psi}U_{\psi}=U_{\psi}J_{\Psi}$.
\end{itemize}
\end{thm}
\begin{proof}
We sketch the proof which is perhaps well--known to the experts, see e.g. Theorem A.13 of \cite{Su} for a more general situation.

(i) Consider $Z:=\spec(\gz)$, and by $\m$ the Radon measure (i.e. the {\it basic measure}, see \cite{DX}, Section I.7.2)
induced by the vector $\Om$, which is separating for $\gz$. We have 
$\gz\sim C(Z)\sim L^\infty(Z,\di\m)$. By reasoning as in part 1 of Proposition 3.1 of \cite{VW}, we can define a state $\om_z$ for each $z\in Z$. As $\gz=\pi_\om(\ga)'\cap\{U_\om\}'$, it is easy to show that the states $\om_z$ are invariant under $\a$. We decompose $\ch_\om$ by using  
$$
\G_0:=\{\pi_z(A)\Om_z\mid A\in \ga\}\subset\prod_{z\in Z}\ch_z
$$
as the generating fields of square--integrable vectors, see e.g. \cite{DX, T, W}. Here, 
$\big(\ch_z,\pi_z,U_z,\Om_z\big)$ is the GNS covariant representation of $\om_z$. Then first
$\ch_\om(A)=\int^\oplus_Z \ch_z\di\m(z)$, and it is easy to verify that $\{\pi_z\}_{z\in Z}$ is a measurable field of representations, $\{U_z\}_{z\in Z}$ a measurable field of unitary operators, and $\{\Om_z\}_{z\in Z}$ is a measurable field of unit vectors such that 
$\pi_\om(A)=\int^\oplus_Z \pi_z(A)\di\m(z)$, $A\in\ga$, $U_\om=\int^\oplus_Z U_z\di\m(z)$, and 
$\Om=\int^\oplus_Z \Om_z\di\m(z)$. The map $z\in Z\mapsto f(z):=\om_z\in\cs(\ga)$
is measurable and induces on $\cs(\ga)$ the Radon measure 
\begin{equation*}
\n(A):=\m\circ f^{-1}(A)\,,\quad A\in\cb_0\,.
\end{equation*}
Here $\cb_0$ denotes the $\s$--algebra of the Baire subsets of $\cs(\ga)$, the latter equipped with the $*$--weak topology.
Such a measure is nothing but the orthogonal measure corresponding to 
$\gz\subset\pi_\om(\ga)'$. As all the states $\om_z$ are invariant, $\n$ is automatically supported on the set of invariant states $\ce_\bz(\ga)$, and pseudo--supported on the ergodic states $\partial\ce_\bz(\ga)$ as $\gz=\pi_\om(\ga)'\cap\{U_\om\}'$. We then obtain again the decomposition \eqref{dcpz} for the state $\om$. Notice that, at this stage we do not have used any separable assumption for $\ga$, neither $\gz\subset\gz_\om$, or $\s(\om)\in Z(\ga^{**})$.

(ii) and (iii) Suppose now that $\ga$ is separable, and in addition $s(\om)\in Z(\ga^{**})$,  
$\gz\subset\gz_\om$. First $\n$ is supported on the ergodic states $\partial\ce_\bz(\ga)$ because it is a Baire set.
One shows that, $J_\Om$ is decomposable with decomposition
$J_\Om=\int^\oplus_{\partial\ce_\bz(\ga)}J(\psi)\di\n(\psi)$, because by \eqref{jei} it can be considered as an anti--linear operator commuting with $\gz$. Then as $J_\Om\Om=\Om$, there exists a measurable subset $F_1$ of full measure such that
$$
J_\Psi\Psi=\Psi\,,\quad \psi\in F_1\,.
$$
On the other hand, by using Lemma \ref{dixx}, $\pi_\om(\ga)''
=\int^\oplus_{\partial\ce_\bz(\ga)}\pi_\psi(\ga)''\di\n(\psi)$, 
$\pi_\om(\ga)'=\int^\oplus_{\partial\ce_\bz(\ga)}\pi_\psi(\ga)'\di\n(\psi)$,
$\gz_\om=\int^\oplus_{\partial\ce_\bz(\ga)}\gz_\psi\di\n(\psi)$.
Then for each $X\in\pi_\om(\ga)''$ with $Y=J_\Om XJ_\Om\in\pi_\om(\ga)'$, we get
$J_\Psi X(\psi)J_\Psi=Y(\psi)$, a.e. Fix a countable selfadjoint subsets $\{X_i\}_{i\in\bn}$, 
$\{Y_i\}_{i\in\bn}$, and $\{C_i\}_{i\in\bn}$ such that  $\{X_i(\psi)\}_{i\in\bn}$, 
$\{Y_i(\psi)\}_{i\in\bn}$, and $\{C_i(\psi)\}_{i\in\bn}$ are generating $\pi_\psi(\ga)''$, $\pi_\psi(\ga)'$, and $\gz_\psi$ respectively a.e. We find two measurable subsets $F_2$ and $F_3$ of full measure such that 
$$
J_\Psi X_i(\psi)J_\Psi=Y_i(\psi)\,,\quad i\in\bn
$$
simultaneously for $\psi\in F_2$, and
$$
J_\Psi C_i(\psi)J_\Psi=C^*_i(\psi)\,,\quad i\in\bn
$$
simultaneously for $\psi\in F_3$.  This means that, on $F_2$, 
$J_\Psi\pi_\om(\ga)''J_\Psi=\pi_\om(\ga)'$, and on $F_3$, $J_\Psi CJ_\Psi=C^*$, $C\in\gz_\psi$.
As $J_\Om U_\Om=U_\Om J_\Om$ (cf. Proposition 3.3 of \cite{NSZ}), we have 
$J_\Psi U_\psi=U_\psi J_\Psi$ on a measurable set $F_4$ of full measure. Thus, for   
$\psi\in F:=\bigcap_{i=1}^4 F_i$, $s(\psi)\in Z(\ga^{**})$, $J(\psi)=J_\Psi$, and finally $J_\Psi$ commutes with $U_\psi$.
\end{proof}
Notice that, it is possible to decompose simultaneously also the modular operator, obtaining
$\D_{\Om}=\int^\oplus_{\partial\ce_\bz(\ga)}\D_\Psi\di\n(\psi)$, see e.g. Appendices A, B of \cite{BF}, and the references cited therein.

The reduction theory in non separable situation presents several pathologies, see for example 
\cite{VW}. In addition, the proof of the key result relative to the decomposition of $\gz'$ in part (2) of Proposition 3.1 of \cite{VW} appears not proved, see \cite{Sc} for a connected counterexample. On the other hand, in non separable case the decomposition of $\gz'$ presented in \cite{Te} does not guarantees that the decomposition of the cyclic vector $\Om$ for $M\subset\gz'$, which is fiber--wise cyclic for $\gz'$ by construction, is still fiber--wise cyclic when restricted to any sub algebra $M$ as above. However, it is of interest to prove the following conjecture for the general non separable situation.
\begin{con}
\label{cori}
Let $\ga$ be a unital $C^*$--algebra, and $\om\in\cs(\ga)$ with $s(\om)\in Z(\ga^{**})$. Suppose that $\gz\subset\gz_\om$. Then the orthogonal measure $\n$ in \eqref{dcpz} corresponding to 
$\gz$ is pseudo--supported on the set of states 
$\{\f\in\cs(\ga)\mid s(\f)\in Z(\ga^{**})\}$.
\end{con}
The reader is also referred to \cite{R} for some connected results.

\section{a nonconventional ergodic theorem: ergodic case}

We now pass to the natural application of Theorem \ref{fu1} relative to generic measures, 
to "quantum diagonal measures" arising from an ergodic dynamical system. 
We start with two integers $k_1,k_2\in\bz$, and avoid the trivial cases $k_1=0$, 
$k_1=k_2$ which can be managed by von Neumann Mean Ergodic Theorem, and $k_2=0$ managed by Kovacs--Sz\"ucs Ergodic Theorem (see e.g. \cite{BR}, Proposition 4.3.8). Denote by $\bg\subset\bz\times\bz$ the set of the remaining integers describing the non trivial situations.
Let $\big(\ga,\a,\om\big)$ be $C^*$--dynamical system based on the action of $\bz$. We also drop some subscripts when it is no matter of confusion. Denote $M:=\pi_\om(\ga)''$, $M'=\pi_\om(\ga)'$.  
Suppose further that the support $s(\om)$ in $\ga^{**}$
is central. Let $\gam:=M\ots{max}M'$ be the completion of the algebraic 
tensor product $\gn:=M\odot M'$ w.r.t. the maximal $C^*$--norm (cf. 
\cite{T}, Section IV.4). It is easily seen that, on $\gam$ the 
following two states are automatically well--defined. The first one is 
the canonical product state 
$$
\f(A\otimes B):=\langle A\Om,\Om\rangle\langle B\Om,\Om\rangle\,,
\quad A\in M\,,\,\,B\in M'\,.
$$
The second one is uniquely defined by
$$
\r(A\otimes B):=\langle AB\Om,\Om\rangle\,,
\quad A\in M\,,\,\,B\in M'\,.
$$
The state $\r$ can be considered the (quantum analogue of the) ``diagonal 
measure'' of the ``measure'' $\f$. Fix $k_1,k_1\in\bg$.
On $\gam$ is also defined the automorphism
$$
\g:=\ad{}\!_{U_\om^{k_1}}\otimes\ad{}\!_{U_\om^{k_2}}\,,
$$
see \cite{T}, Proposition IV.4.7. Of course, $\big(\gam,\g,\f\big)$ is 
a $C^*$--dynamical system whose GNS covariant representation is 
precisely 
$\big(\ch_\om\otimes\ch_\om,\id\otimes\id,U_\om^{k_1}\otimes U_\om^{k_2},\Om\otimes\Om\big)$. 
Notice that the 
$*$--subalgebra $\gn$ is globally stable under the action of $\g$.
In addition, again by Proposition IV.4.7 of \cite{T}, 
$$
\s(A\otimes B):=AB\,,
\quad A\in M\,,\,\,B\in M'\,.
$$
uniquely defines a representation of $\gam$ on $\ch_\om$ such that 
$\big(\ch_\om,\s,\Om\big)$ is precisely the GNS representation $\big(\ch_\r,\pi_\r,\rm{P}\big)$ of $\r$.
For $(k_1,k_2)\in\bg$, define
$$
\S(k_1,k_2):=\{(v,w)\in\s_{\mathop{\rm pp}}
(U_\om)\times\s_{\mathop{\rm pp}}(U_\om)\,:\,
v^{k_1}w^{k_2}=1\}\,.
$$
Then by Lemma 4.18 of \cite{Fu1},
\begin{equation}
\label{tspd}
E^{U_\om^{k_1}\otimes U_\om^{k_2}}_{1}=\bigoplus_{s\in\S(k_1,k_2)}E^{U_\om}_{v_{s}}\otimes E^{U_\om}_{w_{s}}\,.
\end{equation}
In addition, consider the series
\begin{equation}
\label{addit1}
V=\sum_{s\in\S(k_1,k_2)}\langle\,{\bf\cdot}\,,V_{v_s}\Om\otimes W_{w_s}\Om\rangle
V_{v_s}W_{w_s}\Om\,,
\end{equation}
where $V_{v}\in M$, $W_{w}\in M'$ (equivalently $V_{v}\in M'$, $W_{w}\in M$) with 
$U_\om V_{v}U_\om^*=vV_{v}$, $U_\om W_{w}U_\om^*=wW_{w}$ with $z^{k_1}w^{k_2}=1$.
\begin{prop}
\label{diam1}    
Let $\big(\ga,\a,\om\big)$ be a $C^*$--dynamical system such 
that $s(\om)\in Z(\ga^{**})$. Fix $(k_1,k_2)\in\bg$ and suppose that 
$U_\om^{k_2-k_1}$ is ergodic. Then the following assertions hold true.
\begin{itemize}
\item[(i)] The state $\r\in\cs(\gam)$ is 
generic for $\big(\gam,\g,\f\big)$ w.r.t. $\gn$.
\item[(ii)] The sum in \eqref{addit1} converges in the strong operator topology and provides the explicit expression for the isometry $V$ (depending on $k_1,k_2$) in \eqref{fu4}.
\item[(iii)] We get for the Tomita conjugation $J_\Om$,
\begin{equation}
\label{tomi}
J_\Om V=VJ_\Om\otimes J_\Om\,.
\end{equation}
\end{itemize}
\end{prop}
\begin{proof}
The proof follows the lines of Proposition 4.1 of \cite{F3}. 

(i) Consider the linear generator $A\otimes B$ of $\gn$, with $A\in M$ and $B\in M'$. Then by the von Neumann Ergodic Theorem,
\begin{align*}
\frac{1}{N}\sum_{n=0}^{N-1}\psi(\g^{n}(A&\otimes B))
=\frac{1}{N}\sum_{n=0}^{N-1}\langle AU^{n(k_2-k_1)}B\Om,\Om\rangle\\
&\equiv\bigg\langle A\bigg(\frac{1}{N}\sum_{n=0}^{N-1}U^{n(k_2-k_1)}\bigg)B\Om,
\Om\bigg\rangle\\
&\longrightarrow\langle A\Om,\Om\rangle\langle B\Om,\Om\rangle
\equiv\f(A\otimes B)\,.
\end{align*}
As $E^{U^m}_1\geq E^{U}_1$, 
if $U^{k_2-k_1}$ is 
ergodic, $U$ is 
ergodic too, and by Proposition \ref{gnsc1} $E^{U}_{v}\ch$ is one dimensional, 
$v\in\s_{\mathop{\rm pp}}(U)$. Thus,
$E^{U}_{v_{s}}\ch$ and $E^{U}_{w_{s}}\ch$ are generated by 
$V_{v_{s}}\Om$, $W_{w_{s}}\Om$, where $V_{v_{s}}$ and $W_{w_{s}}$ are 
unitaries of $M_{v_{s}}$, $(M')_{w_{s}}$ respectively. By \eqref{tspd} 
$E^{U}_{v_{s}}\ch\otimes E^{U}_{w_{s}}\ch$ is one dimensional, and it 
is generated by $V_{v_{s}}\Om\otimes W_{w_{s}}\Om$. 
By taking into account $v_{s}^{k_1}w_{s}^{k_2}=1$, $V_{v_{s}}\otimes W_{w_{s}}$ is invariant under $\g$ for each
$s\in\S$. Then we can conclude that 
$\gn^{\g}\Om\otimes\Om$ is dense in $E^{U}_{v_{s}}\ch\otimes E^{U}_{w_{s}}\ch$. 

(ii) For $\xi\in\ch\otimes\ch$ such that 
$E^{U^{k_1}\otimes U^{k_2}}_{1}\xi$ has finite support $S_\xi$, we easily get
$$
V\xi=\sum_{s\in S_\xi}\langle\xi,V_{v_s}\Om\otimes W_{w_s}\Om\rangle
V_{v_s}W_{w_s}\Om\,.
$$
In addition, being the set of such vectors dense in $\ch\otimes\ch$, for each 
$\xi\in\ch\otimes\ch$ we find one of such vectors $\xi_\eps$ with finite support 
$S_{\xi_\eps}$ such that $\|\xi-\xi_\eps\|<\eps/2$. For each finite subsets $R,S\supset\S(k_1,k_2)$, we get
\begin{align*}
&\big\|\sum_{s\in R}\langle\xi,V_{v_s}\Om\otimes W_{w_s}\Om\rangle
V_{v_s}W_{w_s}\Om
-\sum_{s\in S}\langle\xi,V_{v_s}\Om\otimes W_{w_s}\Om\rangle
V_{v_s}W_{w_s}\Om\big\|\\
\leq&\big\|\sum_{s\in R}\langle\xi,V_{v_s}\Om\otimes W_{w_s}\Om\rangle
V_{v_s}W_{w_s}\Om
-\sum_{s\in R}\langle\xi_\eps,V_{v_s}\Om\otimes W_{w_s}\Om\rangle
V_{v_s}W_{w_s}\Om\big\|\\
+&\big\|\sum_{s\in R}\langle\xi_\eps,V_{v_s}\Om\otimes W_{w_s}\Om\rangle
V_{v_s}W_{w_s}\Om
-\sum_{s\in S}\langle\xi_\eps,V_{v_s}\Om\otimes W_{w_s}\Om\rangle
V_{v_s}W_{w_s}\Om\big\|\\
+&\big\|\sum_{s\in S}\langle\xi_\eps,V_{v_s}\Om\otimes W_{w_s}\Om\rangle
V_{v_s}W_{w_s}\Om
-\sum_{s\in S}\langle\xi,V_{v_s}\Om\otimes W_{w_s}\Om\rangle
V_{v_s}W_{w_s}\Om\big\|\\
\leq&\|VP_R(\xi-\xi_\eps)\|+\|VP_S(\xi-\xi_\eps)\|
\leq2\|\xi-\xi_\eps\|=\eps\,,
\end{align*}
where $P_T$ is the selfadjoint projection onto the closed subspace generated by the eigenvectors corresponding to the finite subset 
$T\subset\S(k_1,k_2)$. 

(iii) Notice that $J$ commutes with $U$ (cf. Proposition 3.3 in \cite{NSZ}). Thus, if 
$\xi\in\big(E_{1}^{U^{k_1}\otimes U^{k_2}}\ch\otimes\ch\big)^{\perp}$, then 
$J\otimes J\xi\in\big(E_{1}^{U^{k_1}\otimes U^{k_2}}\ch\otimes\ch\big)^{\perp}$ as well. It is then enough to check \eqref{tomi} on 
$E_{1}^{U^{k_1}\otimes U^{k_2}}\ch\otimes\ch$, the last being generated by the one dimensional subspace 
$E^{U}_{v}\ch\otimes E^{U}_{w}\ch$ for 
$v,w\in\s_{\mathop{\rm pp}}(U)$ with $v^{k_1}w^{k_2}=1$. In addition,
\begin{equation}
\label{addit}
UJV_{v}JU^*=\bar vJV_{v}J\,,\quad UJW_{w}JU^*=\bar wJW_{w}J\,,
\end{equation}
with $\bar v^{k_1}\bar w^{k_2}=1$ too.
By taking into account Proposition \ref{gnsc1} together with \eqref{addit}, and exchanging the role of 
$M$ and $M'$, we obtain
\begin{align*}
&JVV_{v}\Om\otimes W_{w}\Om=JV_{v}W_{w}\Om
=(JV_{z}J)(JW_{w}J)\Om\\
=&V(JV_{v}\Om)\otimes (JW_{w}\Om)=
V(J\otimes J)(V_{v}\Om\otimes W_{w}\Om)\,.
\end{align*}
\end{proof}
Let $\gx_0:=M+M'$ be the linear span of $M$ and $M'$, and $\gx:=\overline{\gx_0}^{\|\,\,\|}$ its norm closure in $\cb(\ch_\om)$. Here we have the main result of the present section.
\begin{thm}
\label{cccc}
Let $\big(\ga,\a,\om\big)$ be a $C^*$--dynamical system such that $s(\om)\in Z(\ga^{**})$.
Fix $(k_1,k_2)\in\bg$, and suppose that
$\om$ ergodic w.r.t. the action of 
$\{\a^{n(k_2-k_1)}\mid n\in\bz\}$. Then for each $X\in \gx$,
\begin{equation}
\label{sta}
\mathop{\rm s\!-\!lim}_{N\to+\infty}
\frac{1}{N}\sum_{n=0}^{N-1}U_\om^{nk_1}XU_\om^{n(k_2-k_1)}
=V\left(X\Om\otimes\,\cdot\,\right)\,.
\end{equation}
\end{thm}
\begin{proof}
If $A\in M$, \eqref{sta} follows from Proposition
\ref{diam1} and Theorem \ref{fu1}. Let now $A\in M'$. By \eqref{tomi} and Proposition 3.3 in 
\cite{NSZ}, the previous part of the proof leads to
\begin{align*}
\frac{1}{N}\sum_{n=0}^{N-1}U^{nk_1}AU^{n(k_2-k_1)}&\xi=J\frac{1}{N}\sum_{n=0}^{N-1}U^{nk_1}JAJU^{n(k_2-k_1)}J\xi\\
\longrightarrow JV(JA\Om\otimes J&\xi)
=JV(J\otimes J)\left(A\Om\otimes \xi\right)\\
=&V\left(A\Om\otimes\xi\right)\,.
\end{align*}
Thus, \eqref{sta} holds true for $X\in M\cup M'$ and then for its linear span $\gx_0$. Let $X\in\gx$ and choose $X_{\eps}\in \gx_0$ with
$\|X-X_{\eps}\|<\frac{2\eps}{3\|\xi\|}$. Then
\begin{align*}
&\bigg\|\frac{1}{N}\sum_{n=0}^{N-1}U^{nk_1}XU^{n(k_2-k_1)}\xi-V\left(X\Om\otimes\xi\right)\bigg\|\\
\leq\frac{2\eps}{3}&+\bigg\|\frac{1}{N}\sum_{n=0}^{N-1}U^{nk_1}
X_{\eps}U^{n(k_2-k_1)}\xi-V\left(X_{\eps}\Om\otimes\xi\right)\bigg\|\,,
\end{align*}
and the proof follows by the first part.
\end{proof}
As a simple application of the previous results, we get
\begin{cor}
Under the hypotheses of Theorem \ref{cccc},
\begin{align*}
&\lim_{N\to+\infty}\frac{1}{N}\sum_{n=0}^{N-1}\om\left(A_{0}\a^{nk_1}(A_{1})
\a^{nk_2}(A_{2})\right)\\
&=\left\langle V\left(\pi(A_{1})\Om\otimes\pi(A_{2})
\Om\right),\pi(A^{*}_{0})\Om\right\rangle\,.
\end{align*}
\end{cor}
\begin{proof}
The proof follows by Theorem \ref{cccc}, reasoning as in the proof of Corollary \ref{fiberw}.
\end{proof}

\section{a nonconventional ergodic theorem: non ergodic case}
\label{nestz}

The present section is devoted to the Nonconventional Ergodic Theorem for the non ergodic situation. 

We start with a discrete $C^*$--dynamical system $(\ga,\a,\om)$ with $\ga$ separable if it is not otherwise specified, such that 
$s(\om)\in Z(\ga^{**})$ and
$\gz:=\pi_\om(\ga)'\cap\{U_\om\}'\subset\gz_\om$. We use the results of the standard reduction procedure  collected in Section \ref{sec:red}. We decompose the Hilbert space $\ch_\om$ by considering $\gz$ as the set of diagonal operators, obtaining the simultaneous decomposition of $\pi_\om(\ga)''$, $\pi_\om(\ga)'$, $U_\om$ and $J_\Om$ according to Theorem \ref{minsepr}. Suppose that $k_1=k$, $k_2=k+1$, with $k\in\bz\backslash\{-1,0\}$. Thanks to Proposition \ref{diam1}, a field of partial isometries 
$V(\psi):\ch(\psi)\otimes\ch(\psi)\to\ch(z)$ is defined on the measurable subset 
$F\subset\partial\ce_\bz(\ga)$ of full measure appearing in Theorem \ref{minsepr}. It is given by
\begin{align}
\label{addit11}
V(\psi)\eta\otimes\xi=&\sum_{s_\psi\in\S_\psi(k_1,k_2)}\langle \eta\otimes\xi,V_{v_{s_\psi}}
\Psi\otimes 
W_{w_{s_\psi}}\Psi\rangle\nn\\
\times&V_{v_{s_\psi}}W_{w_{s_\psi}}\Psi\,,
\end{align}
with
$$
\S_\psi(k_1,k_2):=\{(v,w)\in\s_{\mathop{\rm pp}}
(U_\psi)\times\s_{\mathop{\rm pp}}(U_\psi)\,:\,
v^{k_1}w^{k_2}=1\}\,.
$$
Define for example $V(\psi)=0$ on the negligible set $\partial\ce_\bz(\ga)\backslash F$.
\begin{lem}
\label{cinqu}
The field of operators in \eqref{addit11} is measurable and defines a partial isometry
$$
V:\int_{\partial\ce_\bz(\ga)}^\oplus\ch_\psi\otimes\ch_\psi\di\n(\psi)
\longrightarrow \int_{\partial\ce_\bz(\ga)}^\oplus\ch_\psi\di\n(\psi)=\ch_\om\,.
$$
written as
\begin{equation*}
V:=\int_{\partial\ce_\bz(\ga)}^\oplus V(\psi)\di\n(\psi)\,.
\end{equation*}
\end{lem}
\begin{proof}
As the field $V(\psi)$ is made partial isometries a.e., it is enough to check that it is measurable. For $A\in\ga$ and $\xi,\eta\in\ch_\om$, define the functions $F^{A,\xi,\eta}_N$, $F^{A,\xi,\eta}$ as follows.
\begin{align}
\label{almos}
&F^{A,\xi,\eta}_N(\psi):=\frac{1}{N}\sum_{n=0}^{N-1}\langle U^n_\psi\pi_\psi(A)U^n_\psi \xi(\psi),
\eta(\psi)\rangle\,,\nn\\
&F^{A,\xi,\eta}(\psi):=\langle V(\psi)\pi_\psi(A)\Psi\otimes \xi(\psi),\eta(\psi)\rangle\,.
\end{align}
Everything in \eqref{almos} is well defined up to a negligible set depending on $\xi,\eta\in\ch_\om$.
The $F^{A,\xi,\eta}_N$ are measurable by construction, and by Theorem \ref{cccc} also the
$F^{A,\xi,\eta}$ are measurable functions as point--wise limit of measurable functions. Let 
$\{A_i\mid i\in\bn\}\subset\ga$ be a norm dense subset. 
The assertion follows from Proposition II.2.1 of \cite{DX} as 
$\{\pi_\om(A_i)\Om\mid i\in\bn\}\subset\ch_\om$ provides a fundamental sequence of measurable vector fields.
\end{proof}
For each $B\in\cb(\ch_\om)$ we define $V_B:\ch_\om\longrightarrow\ch_\om$ as
\begin{equation}
\label{stocazz}
V_B\xi:=V\int_{\partial\ce_\bz(\ga)}^\oplus(B\Om)(\psi)\otimes\xi(\psi)\di\n(\psi)\,,\quad
\xi\in\ch_\om\,.
\end{equation}
By Lemma \ref{almos}, $\{V_B\mid B\in\cb(\ch_\om)\}\subset\gz'$.
Here there is the main theorem of the present section.
\begin{thm}
\label{cccc1}
Let $\big(\ga,\a,\om\big)$ be a $C^*$--dynamical system with $\ga$ separable and 
$s(\om)\in Z(\ga^{**})$. Fix 
$k\in\bz\backslash\{-1,0\}$ and $l\in\bz\backslash\{0\}$ such that $k_1=kl$, $k_2=(k+1)l$.
Suppose that $\pi_{\om}(\ga)'\cap\{U_\om^l\}'\subset \gz_\om$.
Then for each $X\in \gx$, 
\begin{equation*}
\mathop{\rm s\!-\!lim}_{N\to+\infty}
\frac{1}{N}\sum_{n=0}^{N-1}U_\om^{nkl}XU_\om^{nl}
=V_X\,,
\end{equation*}
where $V_X$ is the bounded linear operator given in \eqref{stocazz} corresponding to the unitary $U_\om^l$.
\end{thm}
\begin{proof}
By considering the automorphism $\a^l$ implemented on the GNS representation by $U^l_\om$, it is enough to prove the result for $l=1$ without loss of generality. First fix $X\in\gx_0$ and 
$\xi\in\ch_\om$.
We have by Theorems \ref{minsepr} and \ref{cccc},
$$
f_N(\psi):=\left\|\frac{1}{N}\sum_{n=0}^{N-1}U_\psi^{nk}XU_\psi^{n}\xi(\psi)
-V_X(\psi)\xi(\psi)\right\|^2\longrightarrow0\,, a.e.
$$
In addition, $0\leq f_N(\psi)\leq(2\|X\|\|\xi\|)^2\,,a.e.$
By Lebesgue Dominated Convergence Theorem, we get
\begin{align*}
\lim_{N\to+\infty}\bigg\|\frac{1}{N}\sum_{n=0}^{N-1}U_\om^{nk}XU_\om^{n}\xi
-V_X&\xi\bigg\|^2
=\lim_{N\to+\infty}\int_{\partial\ce_\bz(\ga)}f_N(\psi)\di\m(\psi)\\
=&\int_{\partial\ce_\bz(\ga)}\lim_{N\to+\infty}f_N(\psi)\di\m(\psi)=0\,.
\end{align*}
Let now $\xi\in\ch_\om$, $X\in\gx$ and $\eps>0$, then there exists $X_\eps\in\gx_0$ such that
$\|X-X_\eps\|<\frac{\eps}{3\|\xi\|}$. Notice that
$$
\|V_X-V_{X_\eps}\|=\|V_{X-X_\eps}\|\leq\|X-X_\eps\|\,.
$$
A standard $3\eps$--argument leads to
$$
\bigg\|\frac{1}{N}\sum_{n=0}^{N-1}U_\om^{nk}XU_\om^{n}\xi
-V_X\xi\bigg\|\leq\frac{2\eps}{3}
+\bigg\|\frac{1}{N}\sum_{n=0}^{N-1}U_\om^{nk}X_\eps U_\om^{n}\xi
-V_{X_\eps}\xi\bigg\|\,,
$$
and the proof follows by the first part.
\end{proof}
\begin{cor}
\label{fiberw}
Under the same hypotheses  of Theorem \ref{cccc1}, we get
$$
\lim_{N\to+\infty}\frac{1}{N}\sum_{n=0}^{N-1}\om\left(A_{0}\a^{nkl}(A_{1})
\a^{n(k+1)l}(A_{2})\right)
=\big\langle V_{\pi_\om(A_1)}\pi_\om(A_2)\Om,\pi_\om(A_0^*)\Om
\big\rangle\,.
$$
\end{cor}
\begin{proof}
The proof follows by Theorem \ref{cccc1} by noticing that
$$
\om\left(A_{0}\a^{nkl}(A_{1})\a^{n(k+1)l}(A_{2})\right)
=\big\langle U^{nkl}XU^{nl}\xi,\eta\big\rangle\,,
$$
where $X=\pi_\om(A_1)$, $\xi=\pi_\om(A_2)\Om$ and $\eta=\pi_\om(A_0^*)\Om$ are measurable and suare--integrable fields of operators and vectors.
\end{proof}
The result explained in the present section can be viewed in terms of convex combinations (i.e. direct integral) of diagonal measures. 
As it  plays no role in proving our results, we sketch the construction by leaving the technical details to the reader. 

For a generic $\f\in\cs(\ga)$, put $M_\f=\pi_{\f}(\ga)''$, 
$M_\f'=\pi_{\f}(\ga)''$, and consider the $C^*$--algebra denoted symbolically as 
$$
\gam:=\int^{\oplus}_{\partial\ce_\bz(\ga)}M_\psi\otimes_{{\rm max}} M'_\psi \di\m(\psi)\,.
$$
It is in fact enough to quotient and complete the concrete $*$--algebra
$$
\gam_0:=\bigg\{\int^{\oplus}_{{\partial\ce_\bz(\ga)}}A(\psi)\otimes B(\psi)\di\m(\psi)\mid 
A\in M_\om\,, B\in M'_\om\bigg\}\,,
$$
w.r.t. $\|X\|:=\sup\{\|\pi(X)\|\mid\pi\,\text{representation of}\,\gam_0\}$, see e.g. Definition IV.4.5 of
\cite{T}.
On $\gam$ it  is defined the automorphism $\b$ given on the generators as
$$
\b:=\int^{\oplus}_{\partial\ce_\bz(\ga)}\ad{}\!_{U_\psi^{k_1}}\otimes\ad{}\!_{U_\psi^{k_2}} \di\m(\psi)\,.
$$
The 
following two states are automatically well--defined. 
The first one, automatically invariant w.r.t. $\b$, arises as the convex combination (i.e. direct integral) of product states. It is given on the generators as
$$
\f(AB):=\int_{\partial\ce_\bz(\ga)}\langle A(\psi)\Psi,\Psi\rangle\langle B(\psi)\Psi,\Psi\rangle\di\m(\psi)\,.
$$
The second one, which corresponds to the generic measure for $(\gam,\b,\f)$, is nothing but
the convex combination of diagonal measures. It is given on the set of generators as
$$
\r(AB):=\int_{\partial\ce_\bz(\ga)}\langle A(\psi)B(\psi)\Psi,\Psi\rangle\di\m(\psi)\,.
$$

\section{a counterexample}
\label{xcouxs}

We construct an irreducible, ergodic $C^*$--dynamical system for which the three--point correlations do not converge for some element.
We start with a sequence $a\in\ell^\infty(\bn)$ such that the Cesaro means
$$
M_n(a):=\frac1n\sum_{k=0}^{n-1}a_k
$$
do not converge. It certainly exists just by taking for the $a_k$, sequences of zeroes and ones such that $M_n(a)$ oscillates between two values $0<\a<\b<1$. Extend the sequence $a$ to all of 
$\bz$ by putting $a_{-k}=a_k$ on negative elements. Define the operator $A\in\cb(\ell^2(\bz))$ with matrix--elements
\begin{equation*}
A_{ij}:=a_i\d_{i,-j}\,,\quad i,j\in\bz\,.
\end{equation*}
Consider the shift
$$
Ve_n:=e_{n+1}\,,\quad n\in\bz\,.
$$
By using Fourier Transform $\cf:L^2(\bt)\to\ell^2(\bz)$, we get $\cf^*V\cf=M_f$,
where $M_f$ is the multiplication operator on the unit circle $\bt$, by the function $f(z)=z$. Thus, the shift has absolutely continuous (w.r.t. the Lebesgue measure on $\bt$) spectrum, and 
\begin{equation}
\label{aantidd}
\lim_n\langle V^n\xi,\eta\rangle=0
\end{equation}
by Riemann--Lebesgue Lemma.
\begin{lem}
\label{nconve}
We have
$$
\frac1n\sum_{k=0}^{n-1}\langle V^kAV^ke_0,e_0\rangle=M_n(a)\,.
$$
\end{lem}
\begin{proof}
The assertion follows as
$$
\langle V^kAV^ke_0,e_0\rangle=\langle AV^ke_0,V^{-k}e_0\rangle
=\langle Ae_k,e_{-k}\rangle=A_{k,-k}=a_k\,.
$$
\end{proof}
Let $\ch:=\ell^2(\bz)\oplus\bc e$ where $e$ is an extra unit vector. Define $U:=V\oplus 1$. On 
$\ga:=\cb(\ch)$, put
$\a:=\ad_U$ and $\om:=\langle\,{\bf\cdot}\,e,e\rangle$. The main properties of the $C^*$--dynamical system 
$(\ga,\a,\om)$ are listed below:
\begin{itemize}
\item[(i)] the GNS quadruple for $(\ga,\a,\om)$ is precisely $(\ch\,, \id\,, U\,, e)$;
\item[(ii)]  the system is irriducible: $\pi_\om(\ga)'\equiv\ga'=\bc\idd$;
\item [(iii)] the system is mixing: $\lim_{n\to+\infty}\om(\a^n(X)Y)=\om(X)\om(Y)$, 
$X,Y\in\cb(\ch)$, by \eqref{aantidd}.
\end{itemize}
This implies that $(\ga,\a,\om)$ is ergodic, but $s(\om)\notin\ga^{**}$.
Let $B:=\langle\,{\bf\cdot}\,,e\rangle e_0$.
\begin{prop}
The three--point correlations
$$
\frac1n\sum_{k=0}^{n-1}\om(B^*\a^n(A)\a^{2n}(B))
$$
do not converge.
\end{prop}
\begin{proof}
We get
\begin{align*}
\om(B^*\a^n(A)\a^{2n}(B))=&\langle U^nAU^nBe,Be\rangle
=\langle U^nAU^ne_0,e_0\rangle\\
=&\langle V^nAV^ne_0,e_0\rangle=a_n\,.
\end{align*}
The proof follows by Lemma \ref{nconve} as $M_n(a)$ does not converge by construction.
\end{proof}
We have just shown that we cannot drop the condition $s(\om)\in Z(\ga^{**})$.
In the situation under consideration, 
$\pi_\om(\ga)'\cap\{U_\om\}'$ being Abelian is the natural condition to provide the ergodic decomposition of invariant states.
It is of interest to see whether the conditions 
$s(\om)\in Z(\ga^{**})$ and $\pi_\om(\ga)'\cap\{U_\om\}'\subset\gz_\om$ are optimal. In other words, it is still open to check if there exists a counterexample for which $s(\om)\in Z(\ga^{**})$ and 
$\pi_\om(\ga)'\cap\{U_\om\}'\not\subset\gz_\om$, such that the Cesaro means 
\eqref{tpc}, with $k_1=kl$, $k_2=(k+1)l$, does not converge for some $k,l\in\bz$ and $A_i\in\ga$ in non ergodic situation. If there exist counterexamples for which 
$\pi_\om(\ga)'\cap\{U_\om\}'$ is Abelian but not in $\gz_\om$ such that the three--point correlations do not converge, we can find counterexamples with $\pi_\om(\ga)'\cap\{U_\om\}'$ non Abelian by using the tensor product construction.
However, it is also of interest to check whether the Nonconventional Ergodic Theorem holds true when $s(\om)\in Z(\ga^{**})$ and 
$\pi_\om(\ga)'\cap\{U_\om\}'\sim\bm_2(\bc)$, the last being the simplest case as above which does not arise from an elementary tensor product construction.

\section{Compact dynamical systems}

For the sake of completeness, we analyse the case of compact $C^*$--dynamical systems
$(\ga,\a,\om)$, that is when $\ch_\om$ is generated by the eigenvectors of $U_\om$, see e.g.
\cite{NSZ}. We start with the following
\begin{lem}
\label{comptk}
Let $k_1,k_2\in\bg$. For each $A\in\cb(\ch_\om)$, the series
$$
\sum_{\{v,w\in\s_{\mathop{\rm pp}}(U_\om)\mid v^{k_1}w^{k_2-k_1}=1\}}
E^{U_\om}_vAE^{U_\om}_w\,,
$$
converges in the strong operator topology to an operator $S(A)\in\cb(\ch_\om)$ with 
$\|S(A)\|\leq\|A\|$.
\end{lem}
\begin{proof}
Let  $\{f_j\mid j=0,\dots, |k_2-k_1|-1\}$ be the functions describing the $|k_2-k_1|$ solutions of the equation $v^{k_1}w^{k_2-k_1}=1$ w.r.t. the unknown $w\in\bt$.
By reasoning as in Lemma 2.2 of \cite{F2}, we get
\begin{align*}
\bigg\|\sum_{\{v,w\mid v^{k_1}w^{k_2-k_1}=1\}}
E^{U_\om}_vAE^{U_\om}_w\xi\bigg\|^2
\leq&\|A\|^2\sum_v\bigg(\sum_j\|E^{U_\om}_{f_j(v)}\xi\|^2\bigg)\\
\leq&|k_2-k_1|\|A\|^2\|\xi\|^2\,.
\end{align*}
Thus, the series under consideration is absolutely summable in $\ch_\om$ for each 
$A\in\cb(\ch_\om)$ and $\xi\in\ch_\om$ with norm less than $\|A\|\|\xi\|$ and the proof follows.
\end{proof}
\begin{thm}
Suppose that the $C^*$--dynamical system $(\ga,\a,\om)$ is compact. Then for each 
$A\in\cb(\ch_\om)$,
$$
\mathop{\rm s\!-\!lim}_{N\to+\infty}
\frac{1}{N}\sum_{n=0}^{N-1}U_\om^{nk_1}XU_\om^{n(k_2-k_1)}
=S(A)\,.
$$
\end{thm}
\begin{proof}
By Lemma \ref{comptk} and a standard 3--$\eps$ argument, we can reduce the matter to the case when
$\xi=\sum_w E^{U_\om}_w\xi$ is a finite combination of eigenvectors of $U_\om$. We compute, by taking into account von Neumann Ergodic Theorem,
\begin{align*}
\frac{1}{N}\sum_{n=0}^{N-1}U_\om^{nk_1}AU_\om^{n(k_2-k_1)}&\sum_w E^{U_\om}_w\xi
=\sum_w\bigg(\frac{1}{N}
\sum_{n=0}^{N-1}\big(w^{k_2-k_1}U_\om^{k_1}\big)^n\bigg)AE^{U_\om}_w\xi\\
\longrightarrow&\sum_{\{v,w\mid v^{k_1}w^{k_2-k_1}=1\}}
E^{U_\om}_vAE^{U_\om}_w\xi\equiv S(A)\xi
\end{align*}
\end{proof}

As an immediate consequence, we get for the three--point correlations of compact 
$C^*$--dynamical systems,
$$
\lim_{N\to+\infty}\frac{1}{N}\sum_{n=0}^{N-1}\om\left(A_{0}\a^{nk_1}(A_{1})
\a^{nk_2}(A_{2})\right)
=\left\langle S(\pi(A_{1}))\pi(A_{2})\Om,\pi(A^{*}_{0})\Om\right\rangle\,.
$$

\end{document}